\documentclass[11pt]{amsart}
\usepackage [latin1]{inputenc}
\usepackage{amsmath}
\usepackage{amssymb}
\usepackage{graphicx}
\usepackage{subfigure}

\newtheorem{theorem}{Theorem}[section]
\newtheorem{corollary}[theorem]{Corollary}
\newtheorem{lemma}[theorem]{Lemma}
\newtheorem{proposition}[theorem]{Proposition}
\theoremstyle{definition}
\newtheorem{definition}[theorem]{Definition}

\newtheorem{remark}[theorem]{Remark}

\newtheorem{algorithm}{Algorithm}[section]

\newtheorem{Acknowledgments}{Acknowledgments}

\newcommand{\IR}{{\mathbb{R}}}

\newcommand{\IN}{{\mathbb{N}}}

\newcommand{\IB}{{\mathbb{B}}}
\newcommand{\BE}{\begin{equation}}
\newcommand{\EE}{\end{equation}}

\numberwithin{equation}{section}

\title[Inexact Descent Algorithm for MOs on General Riemannian Manifolds]{Convergence analysis of  inexact descent algorithm for multiobjective  optimizations on   Riemannian manifolds without curvature constraints}

\author[X. M. Wang]{Xiangmei Wang}
\address[X. M. Wang]{College of Science, Guizhou University, Guiyang 550025, P. R. China}
\email{{\tt xmwang2@gzu.edu.cn}}

\author[J. H. Wang]{Jinhua Wang}
\address[J. H. Wang]{Department of Mathematics,
 Hangzhou Normal University, Hangzhou 311121, P. R. China}
\email{\tt wangjh@hznu.edu.cn}

\author[C. Li]{Chong Li}
\address[C. Li]{Department of Mathematics, Zhejiang University, Hangzhou 310027,P. R. China}
\email{\tt cli@zju.edu.cn}



\begin{document}

\begin{abstract} We study the  convergence issue for inexact descent algorithm (employing general step sizes) for multiobjective optimizations on general Riemannian manifolds  (without curvature constraints). Under the assumption of the local convexity/quasi-convexity, local/global convergence  results are established. On the other hand, without the assumption of the local convexity/quasi-convexity, but under a  Kurdyka-{\L}ojasiewicz-like condition, local/global linear convergence results are presented, which seem new even in   Euclidean spaces setting and improve sharply the corresponding results in \cite{WLWYSIAM2019} in  the case when the multiobjective optimization is reduced to the scalar case. Finally, for  the special case when the  inexact descent algorithm  employing
    Armijo rule, our results
 improve  sharply/extend  the corresponding ones in \cite{Bento2013M,Bento2012M,Wangxm2019}. 
\end{abstract}

\maketitle


\section{Introduction}
Let $F:\IR^m\rightarrow\IR^n$ be a vector function defined on $\IR^m$.
The multicriteria optimization problem consists of minimizing several objective functions simultaneously, which
 is formulated as follows:
\begin{equation}\min_{x\in \IR^m}F(x).\label{eq-vectoro}\end{equation} Since there is usually no single point which will minimize all given objective functions
simultaneously, the concept of Pareto-optimality or efficiency is considered in stead of the concept of optimality.
 Recall from \cite{Fliege2000M,HuCWLYSiam} that a point $p\in \IR^m$ is called a Pareto point of \eqref{eq-vectoro} 
(or an efficient point), if there does not exist a different point $q\in \IR^m$ such that $F(q)\preceq F(p)$ and $F(q)\neq F(p)$ (where sign ``$\preceq$" means the classical partial order on Euclidean space $\IR^n$; see \eqref{ptd} in Section 2 for the definition.)

 Problem \eqref{eq-vectoro}
 arises in many applications such as engineering disciplines, location science, statistics, management science; see, e.g., [5,6,13,24] and the references therein. 
One of the standard techniques for finding the Pareto points  of \eqref{eq-vectoro} is the scalarization approach, which in fact tries to compute a discrete approximation to the whole set of the Pareto points. Since it was proposed by Geoffrion in \cite{Geoffrion1968} for solving the multicriteria optimization problems in Euclidean spaces, 
the scalarization technique has been extensively studied in the literature; see, e.g., \cite{DasDennis1998, Drummond2008,Jahn1984,Luc1989,Fonseca1995,Miettinen1999} for more details. 
In general, the scalarization approach requires some parameters to be specified in advance, leaving
the modeler and the decision-maker with the burden of choosing them. Another important approach for finding the Pareto points is the descent-type method.
This type of method usually does not require any parameter information, which includes
such as the (steepest) descent algorithm, Newton method, proximal point method and trust-region method; see, e.g., \cite{Fliege2000M,Fliege2009SIAM,DrummondI2004,Drummond2005M,FukudaD2011,FukudaD2013,BonnelIS2005,ChenHY2005,Ryu2014SAIM}.
We are particularly interested in the (steepest) descent algorithm proposed by Fliege and Svaiter in \cite{Fliege2000M} for solving the multicriteria optimization problem in Euclidean spaces, 
which was well-studied and has been extended to the multiobjective optimization (equipped with the partial order induced by a general closed convex pointed cone); see. e.g., \cite{DrummondI2004,Drummond2005M,FukudaD2011,FukudaD2013} and the references therein. 

Recently,
some important notions, techniques and approaches in Euclidean spaces have been
extended to Riemannian manifold settings; see, e.g., \cite{Ferreira2005,LiMWY2011,LiY2012,Mahony1996,Miller2005,WLWYSIAM2015} and the references therein. As pointed out in \cite{Bento2013M}, such extensions are natural and, in general, nontrivial; and enjoy some important advantages; see, e.g., \cite{Absil2008,Smith1994,Udriste1994,Yang2007,WLWYSIAM2019} for more details.
In particular, in \cite{WLWYSIAM2019},  the gradient algorithm (employing general step sizes) was extended  for scalar  optimization problems on general Riemannian manifolds  (without curvature constraints). Under the assumption of the local convexity/quasi-convexity (resp. weak sharp minima), local/global convergence  (resp. linear convergence) results are established (see \cite{WLWYSIAM2019}).

One the other hand,,
the exact/inexact 
descent algorithm  employing
    Armijo rule  was recently extended to solve the multicriteria optimization problem on Riemannian manifolds in \cite{Bento2012M,Bento2013M}, 
where it was shown that the partial convergence property (i.e., each cluster point of the generated sequence by the  inexact descent algorithm is a Pareto critical point) holds  on general Riemannian manifolds, while the full convergence does for the (vector) objective
function being quasi-convex on the whole manifold   of nonnegative
sectional curvatures; see \cite[Theorems 5.1 and 5.2]{Bento2013M}.  The further development of  this   full convergence results  of  the exact/inexact 
descent algorithm  employing
    Armijo rule have been given in \cite{Wangxm2019} where they were  established    under the following weaker  assumption
    \begin{itemize}
    \item [{\bf (A)}]
      the objective function   is quasi-convex only on a sub-level set which is of curvatures bounded from below.
\end{itemize}

The main purpose of the present paper is to   study the local/global convergence issue for the inexact descent algorithm (employing general step sizes) for multiobjective optimizations on general Riemannian manifolds  (without curvature constraints).
The present paper contains two topics  of convergence results for the descent algorithm
employing more general step sizes (which includes the Armijo step sizes   as a special case).

One is  the local/global convergence  for locally quasi-convex function $F$  which includes local convergence, that is,  any sequence generated with initial point close enough to a critical point     converges   to  a critical point (see Theorem  \ref{Local-Convergence}, which seem  new   in the linear space setting),  and the   global convergence   which means that   any sequence    generated with arbitrary initial point  from the domain of the function $F$ does (see Theorem \ref{full-1}(i) and Corollary  \ref{full-AG}).  In particular,   the global convergence result is established  for the descent  algorithm employing  the Armijo step sizes    under the following weaker assumption than (A)   (see Lemma \ref{lemma-BX}):
\begin{itemize}
    \item [{\bf (H)}] The generated sequence $\{p_k\}$   has a cluster point $\bar p$   and $F$ is quasi-convex around $\bar p$.
\end{itemize}

The other is    the locally/globally linear  convergence  without locally quasi-convex assumption for $F$  which includes local convergence, that is,  any sequence generated with initial point close enough to a weak  Pareto  optimum    converges   to  a weak  Pareto  optimum(see Theorem  \ref{Local-Linear}, which seems  new   in the linear space setting in the case when the    Kurdyka-{\L}ojasiewicz-like  property holds at the weak  Pareto  optimum),  and the   global convergence   which means that   any sequence    generated with arbitrary initial point  from the domain of the function $F$ does (see Theorem \ref{full-1}(ii) and Corollary  \ref{full-AG}), that is,
   if  the following  assumption is   assumed,  we   show    that the sequence $\{p_k\}$ converges linearly:
\begin{itemize}
   \item The generated sequence $\{p_k\}$   has a cluster point $\bar p$ which is a locally weak  Pareto  optimum,  the Kurdyka-{\L}ojasiewicz-like  property holds  at     $\bar p$   and the step sizes $\{t_k\}$ has a positive lower bound.
\end{itemize}
(Note by Lemma \ref{Bt-Lip} that the Armijo step sizes has a positive lower bound if Jacobian  $JF$ is Lipschitz continuous around $\bar p$).
 To the best of our knowledge, this global  linear convergence result also seems  new even  in the linear space setting.

Note that our results in the present paper extend/improve the corresponding results in \cite{WLWYSIAM2019} for scalar optimization problems on Riemannian manifolds to multiobjective optimizations on Riemannian manifolds. In particular, it should be remarked that for the linear convergence of the gradient method,
 our result  improves sharply the corresponding result in \cite{WLWYSIAM2019} in the sense that we remove  the local  quasi-convexity assumption;
see Remark \ref{extendgra}.

The remaining of the paper is organized as follows. Some basic notions and notation on Riemannian manifolds and  the inexact descent algorithm employing general step sizes  for solving the multicriteria problem on Riemannian manifolds  are presented in the next section. 
In Section 3,  some related properties about the convexity properties
of vector  functions and some  useful lemmas are presented,  and
local  convergence  results are established, while   locally linear convergence result is presented in Section 4. Global  convergence  (resp. linear convergence) results are presented in the last section.  

\section{Preliminaries and inexact descent algorithm}

\subsection{Notation and notions on Riemannian manifolds}\ \

The notation and notions on Riemannian manifolds used in the present paper are standard, and the readers are referred to some textbooks for more details; see, e.g.,
\cite{Carmo1992,Sakai1996,Udriste1994}.

Let $M$ be a connected and complete $m$-dimensional Riemannian manifold. We use
$\nabla$ to denote the Levi-Civita connection on $M$. Let $p\in M$,
and let $T_{p}M$ stand the tangent space at $p$ to $M$. We denote by
$\langle,\rangle_{p}$ the scalar product
 on $T_{p}M$ with the associated norm $\|\cdot\|_{p}$, where the subscript $p$ is
sometimes omitted. For $q\in{M}$, let $\gamma:[0,1]\rightarrow M$ be
a piecewise smooth curve joining $p$ to $q$. Then, the arc-length of
$\gamma$ is defined by $l(\gamma):=\int_{0}^{1}\|{\gamma}'(t)\|dt$; and the Riemannian distance from $p$ to $q
$ is defined by ${\rm
d}(p,q):=\inf_{\gamma}l(\gamma)$, where the infimum is taken over
all piecewise smooth curves $\gamma:[0,1]\rightarrow M$ joining $p$
to $q$. 
A smooth curve $\gamma$ is called a
geodesic 
if and only
if $\nabla_{{\gamma}'}{{\gamma}'}=0$. A geodesic joining $p$ to $q$ is said to be minimal if its
arc-length equals the Riemannian distance between $p$ and $q$. By
the Hopf-Rinow theorem \cite{Carmo1992}, $(M,{\rm d})$ is a complete
metric space, and there is at least one minimal geodesic joining $p$
to $q$. The closed metric ball in $M$ centered at the point $p$
with radius $r>0$ is denoted by $\mathbb{B}(p,r)$, i.e.,
$$
\mathbb{B}(p,r):=\{q\in M:{\rm d}(p,q)\leq r\}.
$$
Let $Q\subseteq M$ be a subset and $p,\,q\in Q$. The set of
all geodesics $\gamma:[0,1]\rightarrow M$ with $\gamma(0)=p$ and
$\gamma(1)=q$ satisfying $\gamma([0,1])\subseteq Q$ is denoted by   $\Gamma^Q_{pq}$, that is,
$$\Gamma^Q_{pq}:=\{\gamma:[0,1]\rightarrow Q:\;\gamma(0)=p,\, \gamma(1)=q\mbox{ and } \nabla_{{\gamma}'}{\gamma}'=0\}.
$$
Recall the convexity radius $r_{\rm cvx}(p)$ of $p\in M$ which is defined by
\begin{equation}\label{convexity-radius}
r_{\rm cvx}(p):=\sup\left\{r>0:\begin{array}{ll}&\mbox{each ball in }\IB(p,r) \mbox{ is strongly convex}\\
&\mbox{and each geodesic in } \IB(p,r)\mbox{ is minimal}\end{array}\right\}.
\end{equation}
Then, $r_{\rm cvx}(p)>0$ for any $p\in M$; see, e.g., \cite[Theorem 5.3]{Sakai1996}.

Definition \ref{convexset} below presents the notions of different kinds of convexities about subsets in $M$;
see e.g., \cite{LiLi2009,Wang2010}.

\begin{definition}\label{convexset}
A nonempty subset $Q$ of the Riemannian manifold $M$ is said to be

{\rm (a)} weakly convex if and only if, for any $p,q\in Q$, there is a minimal
geodesic of $M$ joining $p$ to $q$ and it is in $Q$;

{\rm (b)} 
totally convex if and only if, for any $p,q\in Q$, all geodesics of $M$ joining $p$ to $q$ lie in $Q$.
\end{definition}

Note by definition that the  strong/total  convexity  implies the weakly convexity for any subset $Q$. 

\subsection{Convexity}\ \

Below, we recall the notion of convexity of
  a  real-valued scalar  function $f:M\rightarrow{\IR}$.
Item (b) in the following definition was known in \cite[Definition 6.1 (b)]{LiMWY2011} (for the convexity) and \cite[Definition 2.2]{Papa2008}
(for the quasi-convexity). 

\begin{definition}\label{convexfunction}
Let $f:M\rightarrow {\IR}$
and let $Q\subseteq M$ be weakly convex.  Then, $f$ is said to be

{\rm (a)} convex (resp.  quasi-convex) on $Q$ if,
for any $x, y\in Q$ and any geodesic $\gamma\in
\Gamma^Q_{xy}$, the composition $f\circ\gamma:[0,1]\rightarrow\IR$ is convex (resp.   quasi-convex) on $[0,1]$;

{\rm(b)}   pseudo-convex on $Q$ if $f$ is differentiable and for any $p,q\in Q$, any geodesic $\gamma\in
\Gamma^Q_{pq}$, there holds:
$$\langle\nabla f(p),\gamma'(0)\rangle\ge 0\quad\Longrightarrow\quad f(q)\ge f(p).$$

{\rm (c)} 
%
convex  (resp.  quasi-convex, pseudo-convex) if  $f$  is   convex  (resp.   quasi-convex, pseudo-convex) on $M$.

{\rm (d)}  convex  (resp.  quasi-convex, pseudo-convex) around $x\in M$ if 
$f$ is convex  (resp.   quasi-convex, pseudo-convex) on $\IB(x,r)$ for some $r>0$.
\end{definition}

It is clear that the convexity implies the quasi-convexity and  pseudo-convexity (assuming $f$ is differentiable). The assertions in the following lemma can be proved directly by definition and are known for some special cases; see. e.g., \cite[Theorems 5.1, 6.2]{Udriste1994} for assertion  (i)  and \cite[Proposition 3.1]{Nemeth1998} for assertion (ii).

\begin{lemma}\label{QC-F}
Let $f:M\rightarrow {\IR}$ be differentiable. Let $Q $ be weakly convex and let $x\in Q$. 
Then, the following assertions hold.

{\rm (i)} If $f$ is convex on   $Q$, then it holds  for any 
$y\in Q$ that
$$
f(y)\ge f(x)+\langle\nabla f(x),\gamma_{xy}'(0)\rangle\quad\mbox{for all }\gamma_{xy}\in\Gamma_{xy}^Q.
$$

{\rm (ii)}
If $f$ is quasi-convex on   $Q$, then it holds for any 
$y\in Q$ with $f(y)\le f(x)$ that
$$
\langle\nabla f(x),\gamma_{xy}'(0)\rangle\le0\quad\mbox{for all }\gamma_{xy}\in\Gamma_{xy}^Q.
$$


\end{lemma}

Below, we extend  the notions of  different kinds of convexities to  vector functions on $M$, which are known for the case when $Q=M$; see, items (a), (b) in \cite[definition 5.1]{Bento2012M}
and item (c) in \cite[definition 5.1]{Bento2013M}. To proceed, as usual, we use ``$\preceq$" and ``$\prec$" to denote the classical  partial order and the strictly partial order defined by
\begin{equation}\label{ptd}
\mbox{$x\preceq y$ (or $y\succeq x$) $\Longleftrightarrow$ $y-x\in \IR_{+}^n$ \quad for $x,y\in\IR^n$}
\end{equation}
and
\begin{equation*}\label{sptd}
\mbox{$x\prec y$ (or $y\succ x$) $\Longleftrightarrow$ $y-x\in \IR_{++}^n$ \quad for $x,y\in\IR^n$,}
\end{equation*}
respectively, where $$\IR_{+}^{n}:=\{x=(x_i)\in \IR^{n}:x_i\geq 0,i\in I\}$$ and $$\IR_{++}^{n}:=\{x=(x_i)\in \IR^{n}:x_i>0,i\in I\}.$$

\begin{definition}\label{CQC}
Let $Q\subseteq M$ be weakly convex. The vector function $F:M\rightarrow\IR^n$ is said to be

{\rm (a)} convex on $Q$ if for any $p,q\in Q$ and any geodesic $\gamma\in
\Gamma^Q_{pq}$, there holds:
$$F(\gamma(t))\preceq(1-t)F(p)+tF(q)\quad\mbox{for any }t\in [0,1].$$

{\rm (b)} quasi-convex on $Q$ if for any $p,q\in Q$ and any geodesic $\gamma\in
\Gamma^Q_{pq}$, there holds:
$$F(\gamma(t))\preceq\max\{F(p),F(q)\}\quad\mbox{for any }t\in [0,1].$$

{\rm (c)} pseudo-convex on $Q$ if $F$ is differentiable and for any $p,q\in Q$, any geodesic $\gamma\in
\Gamma^Q_{pq}$, there holds:
$$JF(p)(\gamma'(0))\nprec 0\quad\Longrightarrow\quad F(q)\nprec F(p).$$

\end{definition}

Clearly for a vector function, the convexity implies both the pseudo-convexity (assuming that $F$ is differentiable) and the quasi-convexity.

Proposition \ref{eq-convex} below  shows the equivalence between the convexity of $F$ and its scalarization. Its proof is easy and so is omitted here.
\begin{proposition}\label{eq-convex} Let $Q\subseteq M$ be weakly convex.
$F:=(f_i)_{i\in I}:M\rightarrow\IR^n$ is  convex  (resp.  quasi-convex, pseudo-convex) on $Q$ if and only if  for each $\{\alpha_i: i\in I\}\subset[0,1]$ with $\sum_{i\in I}\alpha_i=1$, $\sum_{i\in I}\alpha_if_i$ is  convex  (resp.  quasi-convex, pseudo-convex) on $Q$.
\end{proposition}

Furthermore, using the same arguments for proving \cite[Proposition 5.1]{Bento2013M} (for the case when $Q:=M$), one can check the following lemma.

\begin{lemma}\label{Prop1}
Let $F$ be a differentiable vector function. Then, $F$ is quasi-convex on $Q$ if and only if, for any $p,q\in Q$ and any geodesic $\gamma_{pq}\in
\Gamma^Q_{pq}$,
$$F(q)\preceq F(p)\quad\Longrightarrow\quad JF(p)(\gamma_{pq}'(0))\preceq 0.$$
Consequently, $F$ is pseudo-convex implies that it is quasi-convex.
\end{lemma}

The following lemma is useful; see
\cite[Proposition 5.2]{Bento2013M}.

\begin{lemma}\label{remarkp}
If $F$ is pseudo-convex (e.g., convex) (on $M$), then a point $p\in M$ is a
Pareto critical point of $F$ if and only if it is a weak Pareto optimum of \eqref{MP}.
\end{lemma}

\subsection{Multiobjective optimizations on Riemannian manifold}\ \

Below, we consider a vector function $F:M\rightarrow\IR^n$ given by
$$\mbox{$F(p):=(f_i(p))_{i\in I}=(f_1(p),f_2(p),\dots,f_n(p))$ \quad for any $p\in M$},$$
where $I:=\{1,2,\cdots,n\}$ and for each $i\in I$, $f_i:M\rightarrow\IR$ is a function defined on $M$. The vector function $F$ is said to be (continuously) differentiable if each $f_i$ is (continuously) differentiable ($i\in I$). For a continuously differentiable vector function $F$, the Riemannian Jacobian $JF$ and its image at $p\in M$ are respectively denoted by
$$JF(p):=({\nabla}f_i(p))_{i\in I}\quad \mbox{and}\quad{\rm Im}JF(p):=\{JF(p)(v): 
v\in T_pM\},$$
where
\begin{equation}\label{Jocobi}
JF(p)(v):=(\langle{\nabla}f_i(p),v\rangle)_{i\in I}. 
\end{equation}

In the remainder of this paper, we always assume that  $F:=(f_i)_{i\in I}:M\rightarrow\IR^n$ is   continuously differentiable. The vector  optimization problem considered in the present paper is denoted by
\begin{equation}\label{MP}
\min_{p\in M}F(p).
\end{equation} 
Recall that a point $p\in M$ is called a (globally) Pareto (resp. weak Pareto) optimum of \eqref{MP} if there dose not  exist  other point $q\in M$ such that
\begin{equation}\label{Pareto-optimum}
F(q)\preceq F(p)\mbox{ (resp. $\prec$)}\quad\mbox{and}\quad F(q)\neq F(p)
\end{equation}
(see, e.g., \cite{Fliege2000M,HuCWLYSiam} in Euclidean space settings). 
Furthermore, a point $p\in M$ is called a locally Pareto (resp. weak Pareto) optimum of \eqref{MP} if there exists a neighborhood $U\subset M$ of $p$ such that there dose not exist  other point $q\in U\setminus\{p\}$ satisfying \eqref{Pareto-optimum}.


Recall from \cite{Bento2012M,Bento2013M}, that a point $p\in M$ is called a Pareto critical point of $F$ if the image of $JF(p)$ satisfies
\begin{equation*}\label{C-Pareto-P}
{\rm Im}(JF(p))\cap(-\IR^n_{++})=\emptyset.
\end{equation*}
By definition, each (locally) Pareto optimum of $F$ is a Pareto critical point of $F$.


Let $p\in M$ and assume that it is not a Pareto critical point of $F$. By definition, there exists a direction $v\in T_pM$ satisfying
$JF(p)(v)\in-\IR^n_{++}$,
that is, $v$ is a descent direction at $p$. We shall give some notation related to the descent directions of $F$ at $p$. As done in \cite{Bento2013M}, 
we consider the following unconstrained optimization problem on $T_pM$:
\begin{equation}\label{OP-DD}
\min_{v\in T_pM}\alpha_p(v):=\max_{i\in I}\langle\nabla f_i(p),v\rangle+\frac12\|v\|^2.
\end{equation}
Noting that $\alpha_p$ is strongly convex on $T_pM$, problem \eqref{OP-DD} has a unique solution. The solution of problem \eqref{OP-DD} and the associated value are denoted by 
$v(p)$ and $\alpha_p^*$ respectively,
that is,
\begin{equation}\label{Def-vp}
v(p):={\rm argmin}_{v\in T_pM}\alpha_p(v),\quad\alpha_p^*:=\alpha_p(v(p)).
\end{equation}
As pointed out in \cite{Bento2013M}, the vector $v(p)$ is in fact a descent direction at $p$ and always called the steepest descent direction at $p$. Furtheremore, we need the concept of the $\sigma$-approximate steepest descent direction, which can be found in \cite[Definition 4.2]{Bento2013M} (see also \cite[Definition 3.4]{Drummond2005M} for the Euclidean space version).

\begin{definition}\label{simaspdd}
Let $\sigma\in[0,1)$. A vector $v_p\in T_pM$ is said to be a $\sigma$-approximate steepest
descent direction at $p$ if it satisfies
$
\alpha_p(v_p)\le(1-\sigma)\alpha_p^*.
$
\end{definition}

For convenience, for any $p\in M$ and $\sigma\in[0,1)$, we use $D_{\sigma}(p)$ to denote the set of all $\sigma$-approximate steepest
descent direction at $p$. It is clear that for any $\sigma\in[0,1)$, $v(p)\in D_{\sigma}(p)$. The following lemma shows some properties related to the (approximate) steepest descent directions.


\begin{proposition}\label{simaspddp}
Let $p\in M$. The following assertions hold:

{\rm (i)} $v(p)=0$ (or $\alpha_p^*=0$) 
if and only if $p$ is a
Pareto critical point.

{\rm (ii)} There exist $\{\lambda_i: i\in I(p)\}\subset[0,1]$ with $\sum_{i\in I(p)}\lambda_i=1$, such that 
\begin{equation}\label{vpo}
v(p)=-\sum_{i\in I(p)}\lambda_i\nabla f_i(p), 
\end{equation}
where $I(p):=\{i\in I: \langle \nabla f_i(p), v(p)\rangle=\max_{j\in I} \langle \nabla f_j(p), v(p)\rangle\}$;
and the function: $M\ni p \mapsto v(p)\in T_pM$
is continuous on $M$.

{\rm (iii)} If $p\in M$ is not a Pareto critical point
and $v_p\in D_{\sigma}(p)$, 
then there holds
\begin{equation}\label{Def-gps}
\alpha_p(v_p):=\max_{i\in I}\langle\nabla f_i(p),v_p\rangle+\frac{1}{2}\|v_p\|^2<0,
\end{equation}
which particularly implies that $v_p$ is a descent direction. Furthermore, the following relation holds:
\begin{equation}\label{Def-gpss}
\|v_p\|\ge (1-\sqrt{\sigma})\|v(p)\|.
\end{equation}

{\rm (iv)} Let $p\in M$ be a Pareto critical point. Then, for any $\varepsilon>0$, there exists $\delta>0$ such that 
\begin{equation}\label{vpn}
\|v_q\|\le\varepsilon\quad\mbox{for any}\quad v_q\in D_{\sigma}(q), q\in \IB(p,\delta).
\end{equation}

\end{proposition}

\begin{proof}
Assertions (i)-(ii) are known in \cite[Lemmas 4.1, 4.2]{Bento2013M}.
To show assertion (iii), suppose that $p$ is not a Pareto critical point. Then, we see from assertion (i) that $v(p)\neq 0$. First, we show that
\begin{equation}\label{vpnap}
\alpha_p^*=-\frac{1}{2}\|v(p)\|^2.
\end{equation}
Granting this, we get that $\alpha^*_p<0$, and so \eqref{Def-gps} is valid by recalling $\alpha(v_p)\le(1-\sigma)\alpha_p^*$. To show \eqref{vpnap}, by definition of the subindex $I(p)$, there holds
\begin{equation}\label{vpnapv}
\langle \nabla f_i(p), v(p)\rangle=\max_{i\in I} \langle \nabla f_i(p), v(p)\rangle\quad\mbox{for each }i\in I(p).
\end{equation}
Note by \eqref{vpo} that there exist $\{\lambda_i: i\in I(p)\}\subset[0,1]$ with $\sum_{i\in I(p)}\lambda_i=1$ such that $v(p)=-\sum_{i\in I(p)}\lambda_i\nabla f_i(p)$. Then, there holds:
$$
-\|v(p)\|^2=\langle \sum_{i\in I(p)}\lambda_i\nabla f_i(p), v(p)\rangle=\sum_{i\in I(p)}\lambda_i\langle \nabla f_i(p), v(p)\rangle.
$$
In view of \eqref{vpnapv}, we get that $\max_{i\in I} \langle \nabla f_i(p), v(p)\rangle=-\|v(p)\|^2$, and so \eqref{vpnap} holds by definition. Letting $v_p\in D_{\sigma}(p)$, we estimate that
$$
\begin{array}{lll}
-\langle v(p), v_p\rangle+\frac12\|v_p\|^2&=\langle \sum_{i\in I(p)}\lambda_i\nabla f_i(p), v_p\rangle+\frac12\|v_p\|^2\\
&\le \max_{i\in I} \langle \nabla f_i(p), v_p\rangle+\frac12\|v_p\|^2\\
&\le(1-\sigma)\alpha_p^*=-\frac12(1-\sigma)\|v(p)\|^2,
\end{array}
$$
where the last equality is by \eqref{vpnap}.
Then, we have that
$\|v(p)-v_p\|^2\le \sigma\|v(p)\|^2$. This implies \eqref{Def-gpss}, and so
assertion (iii) is shown.

To show assertion (iv), let $p$ be a Pareto critical point. Then, by definition, for any $v\in T_pM$, there exists an index $i_v\in I$ such that $\langle \nabla f_{i_v}(p),v\rangle\ge0 
$.
Since $F$ is continuously differentiable, for any $\varepsilon>0$, there exists $\delta>0$ such that
\begin{equation}\label{dvp}\begin{array}{ll}& \mbox{for any $v\in T_pM$, there exists a index $i_v\in I$ satisfying }\\
&\langle \nabla f_{i_v}(q),P_{q,p}v\rangle\ge-\frac{\varepsilon^2}{2}\quad\mbox{for any}\quad q\in\IB(p,\delta).\end{array}
\end{equation}
Fix $q\in \IB(p,\delta)$ and $v_q\in D_{\sigma}(q)$. Then $P_{p,q}v_q\in T_pM$ and it follows from \eqref{dvp} that there exits $i_{v_q}\in I$ satisfying $\langle \nabla f_{i_{v_q}}(q),P_{q,p}P_{p,q}v_q\rangle\ge-\frac{\varepsilon^2}{2}$.
Thus, in view  of  $v_q\in D_{\sigma}(q)$, we get by definition that  that
$$
0\ge(1-\sigma)\alpha_q^*\ge\alpha_q(v_q)\ge \langle \nabla f_{i_{v_p}}(q),v_q\rangle+\frac12\|v_q\|^2\ge -\frac{\varepsilon^2}{2}+\frac12\|v_q\|^2,
$$
which shows \eqref{vpn}, completing the proof.
\end{proof}

We end this subsection by recalling the concept of  s-compatible  decent direction  at $p$. Recall from \cite[Definition 3.4]{Drummond2005M} (see also \cite{Bento2013M}, where the authors used the notion of the compatible scalarization) that a vector $v\in T_pM$ is said to be {\it s-compatible } at $p$ if there exist
$\{\alpha_i^p: i\in I\}\subset[0,1]$ with $\sum_{i\in I}\alpha_i^p=1$ such that
\begin{equation}\label{lemma2-2}
v=-\sum_{i\in I}\alpha_i^p\nabla f_i(p).
\end{equation}

\subsection{Inexact descent algorithm with  general step sizes for multiobjective optimizations}\ \

Below, we propose an  inexact descent algorithm employing general step sizes for solving problem \eqref{MP}. 

\begin{algorithm}\label{SDAA} (Inexact descent algorithm with  general step sizes)\label{DA}

\noindent {\textbf{Step 0}}. Select $p_{0}\in M$, $\sigma,\beta\in(0,1)$, $R\in[1,+\infty)$ and set $k:=0$.

\noindent {\textbf{Step 1}}. If $p_k$ is a Pareto critical point, then stop;
otherwise select 
$v_k\in D_{\sigma}(p_k)$ and construct the geodesic
$\gamma_{k}$ such that
\begin{equation}\label{GDA-1}
 \gamma_{k} (0)=p_{k}\quad\mbox{and}\quad
 \gamma'_k(0)=v_k.
\end{equation}

\noindent {\textbf{Step 2}}. Select the step size $t_k\in(0,R]$ which  satisfies the following inequality:
\begin{equation}\label{GDA-2}
F(\gamma_k(t_k))\preceq F(p_k)+\beta t_k JF(p_k)(v_k).
\end{equation}

\noindent {\textbf{Step 3}}. Set $p_{k+1}:=\gamma_{k}(t_{k})$, replace $k$ by $k+1$ and go to step 1.
\end{algorithm}

Recall that Algorithm \ref{DA} is said to be well defined if for each $k\in\IN$, there always exists $t_k\in (0,1]$ satisfying
\eqref{GDA-2} in Step 2. Let $\nu\in(0,1)$. Algorithm \ref{DA} is said to employ 
the (generalized) Armijo step sizes (cf. \cite{Bento2013M}) if each step size $t_k$ in Step 2 is chosen by
\begin{equation}\label{GDA-2-A}
t_k:=\max\{\nu^{-i}:i\in\IN,\; F(\gamma_k(\nu^{-i}))\preceq F(p_k)+\beta\nu^{-i} JF(p_k)(v_k)\}.
\end{equation}


Define a mapping $\varphi:M\to  \IR$ by  \begin{equation}\label{eq-varp}\varphi(p):= \sup_{q\in M}\min_{i\in I}(f_i(p)-f_i(q))\quad\mbox{for each } p\in M.\end{equation}
Clearly, $\varphi(p)\ge 0$ for each  $p\in M$ and $\varphi(p)=0$ if and only if $p$ is a weak   Pareto  optimum  of \eqref{MP}. 
Moreover, the following lemma quantifies some properties of the function $\varphi$.

\begin{proposition}\label{phi-p}

{\rm (i)} $\varphi$ is locally Lipschitz continuous on $M$, that is, 
for each $\bar p\in M$,     there exit $\delta>0$ and  $L>0$ such that the function $\varphi$ is Lipschitz continuous on $\IB(\bar p,\delta)$ with modulus $L$:
$$
|\varphi(p)-\varphi(p')\|\le L{\rm d}(p,p')\quad\mbox{ for each } p,p'\in \IB(\bar p,\delta).
$$

{\rm(ii)} Let $p,q\in M$. If $ F(q)\preceq F(p)$, then $\varphi(q)\le\varphi(p)$.

{\rm (iii)} Let $\{p_k\}$ (together with associated sequences  $\{t_k\}$, $\{v_k\}$) be a sequence generated by Algorithm \ref{DA}. 
Then, we have the following estimate
\begin{equation}\label{vk}
\frac{\beta t_k}{2}\|v_k\|^2\le \varphi(p_{k})-\varphi(p_{k+1})\quad\forall k\in\IN.
\end{equation}
\end{proposition}

\begin{proof}

(i). Let  $\bar p\in M$.
Noting that $F$ is continuously differentiable, there exist  $\delta>0$ and  $L>0$ such that
$$
|f_i(p)-f_i(p')\|\le L{\rm d}(p,p')\quad\mbox{ for each } p,p'\in \IB(\bar p,\delta) \;\mbox{ and for each }i\in I.
$$
Fix $p,p'\in \IB(\bar p,\delta)$. Then, it follows that
$$
 f_i(p)-f_i(q)\le f_i(p')-f_i(q)+ L{\rm d}(p,p')\quad\mbox{ for each }  i\in I \;\mbox{ and for each } q\in M,
$$   which implies that
$$
\sup_{q\in M}\min_{i\in I} (f_i(p)-f_i(q))\le \sup_{q\in M}\min_{i\in I} (f_i(p')-f_i(q))+ L{\rm d}(p,p'),
$$ that is, $$ \varphi(p)-\varphi(p')\le L{\rm d}(p,p').$$ With similar technique, we can also check  that
$$ \varphi(p')-\varphi(p)\le L{\rm d}(p,p').$$ Hence, it follows that $$\|\varphi(p)-\varphi(p')\|\le L{\rm d}(p,p'),$$ showing assertion (i).


(ii). It's clearly by definition.

(iii).
By the definition of function $\varphi$, one has that
\begin{equation}\label{sequen-ineq-2}
\begin{array}{lllll}
\varphi(p_{k+1})
 &=\sup_{q\in M}\min_{i\in I}(f_i(p_{k+1})-f_i(q))\\
&\le \sup_{q\in M}\min_{i\in I}(f_i(p_k)+\beta t_k\nabla f_i(p_k)^Tv_k-f_i(q))\\
&\le \sup_{q\in M}\min_{i\in I}(f_i(p_k)-\frac{\beta t_k}{2}\|v_k\|^2-f_i(q))\\
&= -\frac{\beta t_k}{2}\|v_k\|^2+\varphi(p_k),\\
\end{array}
\end{equation}
where the first inequality is by \eqref{GDA-2} and the second inequality thanks to $\nabla f_i(p_k)^Tv_k\le -\frac12\|v_k\|^2$ by \eqref{Def-gps}. 
Hence, \eqref{vk} is seen to hold, completing the proof.

\end{proof}

The following proposition  is about some useful properties of sequence $\{p_k\}$ (together with $\{t_k\}$ and $\{v_k\}$) generated by Algorithm \ref{SDAA},  which includes the partial convergence result for Algorithm \ref{DA} (see assertion (iii) below), while assertion (ii) improves the corresponding results in \cite[Theorem 5.1(i)]{Bento2013M} where \eqref{GA-P1} holds under the assumption that $\{p_k\}$ has a  cluster point.

\begin{proposition}\label{PC}
Algorithm \ref{DA} is well defined and each sequence $\{p_k\}$ generated by Algorithm \ref{SDAA} has the following properties:

{\rm (i)} $\{F(p_k)\}$ is non-increasing monotonically and for any $k\in\IN$:
\begin{equation}\label{D-tpp}
\mbox{${\rm d}(\gamma_k(t),p_k)
\le t\,\|v_k\|$\quad for any $t\in [0,t_k]$}.
\end{equation}

{\rm (ii)}
\begin{equation}\label{GA-P1}
\sum_{k\in\IN} t_k^2\|v_k\|^2<+\infty.
\end{equation}

 {\rm (iii)} If    $\{t_k\}$  has  a positive lower bound or  that $\{t_k\}$  satisfies the Armijo step sizes, then each cluster point of the sequence $\{p_k\}$ is a Pareto critical point of $F$.
\end{proposition}
\begin{proof} The well definedness of
Algorithm \ref{DA}   follows   from  \cite[Proposition 4.1]{Bento2013M}.  By Steps 2 and 3 of Algorithm \ref{DA},  assertion (i) is clear.

By \eqref{vk}, one has
$$\sum_{j=0}^{k}\frac{\beta t_j}{2}\|v_j\|^2\le \varphi(p_{0})-\varphi(p_{k+1})<\varphi(p_{0})<+\infty,$$
 showing  assertion (ii).

To show  assertion (iii),
suppose that  $\{t_k\}$  has  a positive lower bound. Then, it follows from \eqref{GA-P1} that $\|v_k\|\to 0$.
Note by \eqref{Def-gpss}  $
\|v_k\|\ge (1-\sqrt{\sigma})\|v(p_k)\|.
$ Hence, one has that $v(p_k)\to 0$ which, together with Proposition \ref{simaspddp}(i) and (ii),  implies that each cluster point of the sequence $\{p_k\}$ is a Pareto critical point of $F$.
In the case when  $\{t_k\}$  satisfies the Armijo step sizes, the conclusion follows from \cite[Theorem 5.1(ii)]{Bento2013M}. The proof is complete.
\end{proof}

\section{Local convergence under locally quasi-convex assumption}\ \

This section is devoted to establishing local convergence of Algorithm \ref{SDAA}  under locally quasi-convex assumption. Firstly, we need   some  useful lemmas.




The inequality in the following lemma plays an important role in our study.
\begin{lemma}\label{basic-QC}
Let $Q\subseteq M$ be  weakly convex with nonempty interior, let $p\in {\rm int}Q$ and $v\in T_pM$ be a s-compatible vector at $p$.
Let $t\ge 0$ and
$\gamma:[0,+\infty)\rightarrow M$ be the geodesic satisfying
\begin{equation}\label{gsf0}
\gamma(0)=p,\quad \gamma'(0)=v\not=0\quad
\mbox{and} \quad \gamma([0,t])\subset {\rm int}Q.
\end{equation}
Suppose further that the  sectional curvatures on $Q$ are bounded from below by some  $\kappa< 0$, and that   $F$ is quasi-convex on $Q$.
Then the following inequality holds for any $q\in {\rm int}Q$ satisfying
$F(q)\preceq F(q)$:
\begin{equation}\label{basic-QC-1}
\begin{array}{ll}
{\rm d}^2(\gamma(t),q)<
{\rm d}^2(p,q)
 +\frac{3 t^2\|v\|^2}{2\hbar\left(\sqrt{|\kappa|} {\rm
d}(p,q)\right)}\quad\mbox{if }\sqrt{|\kappa|}t\|v\|\le 1.
\end{array}
\end{equation}
\end{lemma}

\begin{proof}
 By assumption,
there exist
$\{\alpha_i^p: i\in I\}\subset[0,1]$ with $\sum_{i\in I}\alpha_i^p=1$ such that
$
v=-\sum_{i\in I}\alpha_i^p\nabla f_i(p).
$ Define a function $f:M\to \IR$ by $$f(\cdot):=\sum_{i\in I}\alpha_i^p  f_i(\cdot).$$ Then  $f$ is   quasi-convex (due to Proposition \ref{eq-convex}) and   differentiable on $Q$, $\nabla f(p)=-v$, $f(q)\le f(p)$.
Hence, \cite[Lemma 2.5]{WLWYSIAM2019} is applicable with $q,p$ in place of $z,x$ to concluding that \eqref{basic-QC-1} holds,
completing the proof.
\end{proof}

The following lemmas is known in \cite[Lemma 2.3]{WangOptim2018}. 

\begin{lemma}\label{lemma3}
Let $\{a_{k}\}$, $\{b_k\}\subset (0,+\infty)$ be two sequences satisfying
\begin{equation}\label{l-d0}
a_{k+1}\leq a_{k}(1+b_k)\quad\mbox{{\rm for all} $k\in\IN$},
\end{equation}
and $\sum_{k=0}^{\infty}b_k<\infty$. Then,
$\{a_{k}\}$ is convergent and so it is bounded.
\end{lemma}

Let $S\subset M$ be a subset. Recall that a sequence $\{p_k\}\subset M$ is said to be {\sl{quasi-Fej\'{e}r convergent}} to $S$ if, for any $s\in S$, there exists a sequence
$\{\varepsilon_{k}\}\subset(0,+\infty)$ satisfying
$\sum_{k=1}^{\infty}\varepsilon_{k}<\infty$ such that
\begin{equation}\label{def-fejer}
{\rm d}^{2}(p_{k+1},s)\leq {\rm d}^{2}(p_{k},s)+\varepsilon_{k}\quad\mbox{for each $k\in\IN$}.
\end{equation}
We end this section with the following lemma, which provides some properties for quasi-Fej\'{e}r convergent sequences 
(see e.g., \cite[Theorem 4.3]{Ferreira1998}).

\begin{lemma}\label{fejer2}
Let $\{p_k\}\subset M$ be a sequence
quasi-Fej\'{e}r convergent to $S$. Then, $\{p_k\}$ is
bounded. If, furthermore, a cluster point $p$ of $\{p_k\}$
belongs to $S$, then $\lim_{k\rightarrow\infty}p_k=p$.
\end{lemma}

For the remainder of the paper, we make the following assumption:
\begin{description}
    \item [{\bf (A$_{sc}$)}] Each vector $v_k$ in $\{v_k\}$ is s-compatible at $p_k$.
\end{description}
To study the local convergence of Algorithm {\ref{DA}}, we further need the following assumption:
\begin{equation}\label{assumption-a}
\begin{array}{ll}
\mbox{
  $\bar p$ is a 
  Pareto critical point and
  $F$ is quasi-convex around $\bar p$.}
\end{array}
\end{equation}
For the following key lemma, recall that $R$ is the constant given at the beginning of Algorithm {\rm\ref{DA}}.

\begin{lemma}\label{lemma1-LC} Suppose  that assumption \eqref{assumption-a} holds.
Then, for any $\delta>0$, there exist $\bar \delta, \hat\delta,c>0$
satisfying $ \bar \delta<\hat\delta<\frac\delta2$ such that,  for any $k\in\IN$,
if $\{p_{j}:0\le j\le k+1\}$   generated  by Algorithm {\rm\ref{DA}} satisfies that
  \begin{equation}\label{generatedset}
     p_0\in\IB(\bar p,\bar\delta)\quad \mbox{and}\quad \{p_{j}:1\le j\le k\}\subset \IB(\bar p, \hat\delta),
\end{equation}
 then one has   that
\begin{equation}\label{lemma1-LC00c}
{\rm d}^2(p_{k+1},q)
\leq
{\rm d}^2(p_k,q)+2R t_k\|v_k\|^2\le {\rm d}^2(p_{0},q)+c{\rm d}(p_{0},q)  
\end{equation}
if     $q\in \IB(\bar p,\hat\delta)$ satisfies  $F(q)\preceq F(p_{k+1})$,
and that
\begin{equation}\label{lemma1-LC00c2}
p_{k+1}\in \IB(\bar p,\hat\delta)\quad\mbox{ if }\,F(\bar p)\preceq F(p_{k+1}).
\end{equation}

\end{lemma}

\begin{proof}
Noting that any closed ball is compact, we have   by \cite[p. 166]{Bishop1964} that the curvatures of the ball
 $\IB(\bar p,r_{\rm cvx}({\bar p}))$ are bounded, where   $r_{\rm cvx}({\bar p})$ is the convexity radius at $\bar p$
    defined in \eqref{convexity-radius}.
Let  $\kappa<0$ be   a lower bound of the curvatures of $\IB(\bar p,r_{\rm cvx}({\bar p}))$. 
Thanks to assumption \eqref{assumption-a}, there exists  $\delta>0$ (using a smaller $\delta$ if necessarily)
 such that $F$ is quasi-convex on $\IB(\bar p,\delta)$ and that
\begin{equation}\label{jiash0}
\delta<\min\left\{1, r_{\rm cvx}({\bar p}),   \frac1{\sqrt{|\kappa|}}\right\}. 
\end{equation}
Furthermore, let $L>0$ be such that $\varphi$ is Lipschitz continuous on $\IB(\bar p,\delta)$ with constant $L$ (recalling Proposition \ref{phi-p}(i)):
\begin{equation}\label{delta-00ll}
\varphi(p)-\varphi(q)\le L{\rm d}(p,q)\quad\mbox{for any }p,q\in \IB(\bar p,\delta).
\end{equation}
Now set $c:=\frac{2RL}{\beta}$ and  choose  $\hat\delta,\bar \delta>0$ be such that \begin{equation}\label{firste-re} \bar \delta<\hat\delta<\frac\delta2,\quad\quad
(\bar \delta+c)\bar \delta\le \hat\delta^2
\end{equation}  and
\begin{equation}\label{local-diameter-1}
   \|v_p\|\le \frac{\beta\delta}{2R}\quad\mbox{for any }p\in\IB(\bar p,\hat\delta)\;\mbox{and }v_p\in D_{\sigma}(p),
\end{equation}
(where existence of $\hat\delta$ of the second item of \eqref{local-diameter-1} is because of Proposition \ref{simaspddp}(iv)).
  To proceed, we verify that the implication \eqref{generatedset}$\Longrightarrow$\eqref{lemma1-LC00c} holds
for any   $k\in\IN$, any $\{p_{j}:0\le j\le k+1\}$  generated by Algorithm {\rm\ref{DA}} and any $q\in \IB(\bar p,  \hat\delta)$ satisfying
 $F(q)\preceq F(p_{k+1})$.
 Granting this and assuming that    $F(\bar p)\preceq F(p_{k+1})$. Then we estimate by \eqref{lemma1-LC00c} (applied to   $\bar p$ in place of $q$
 and noting $ {\rm d}(p_0,\bar p)\le \bar \delta$) that
\begin{equation}\label{lemma1-LC00crb}
{\rm d}^2(p_{k+1},\bar p)
\le {\rm d}^2(p_0,\bar p)+c {\rm d}(p_0,\bar p)\le (\bar \delta+ c)\bar \delta<\hat \delta^2,
\end{equation}
which implies $p_{k+1}\in \IB(\bar p, \hat \delta)$
Then, the triple  $(c,\bar\delta,\hat\delta)$ is as desired.

Thus to complete the proof,
let $k\in\IN$, and let
$\{p_{j}:0\le j\le k+1\}$ be generated by Algorithm {\rm\ref{DA}} to satisfy \eqref{generatedset}.
Fix  $j\in\{0,1,\dots,k\}$, and let $\gamma_{j}$  be the geodesic determined by  \eqref{GDA-1}.
Then,  ${\rm d}(p_{j},\bar p)\le  \hat\delta$ by \eqref{generatedset} and so $\|v_{j}\|\le \frac{\beta\delta}{2R}$ by \eqref{local-diameter-1}.
Therefore it follows from \eqref{D-tpp} that, for any $t\in [0,t_j]$,
$${\rm d}(\gamma_j(t),\bar p)\le {\rm d}(\gamma_j(t),p_{j})+{\rm d}(p_{j},\bar p)< t\,\|v_j\|+\hat\delta\le\frac{\beta\delta}{2}+\frac{1}{2}\delta<\delta,$$
(noting that $t\le t_j\le R$) and then one has that
$$
 \gamma_j({[0, t_j]})\subseteq {\rm int}\IB(\bar p,\delta)\subseteq \IB(\bar p,r_{\rm cvx}({\bar p})).
$$
Now let $q\in \IB(\bar p,\hat\delta)$ be such that  $F(q)\preceq F(p_{k+1})$. Then, we have that
$$
{\rm d}(p_j,q)\le {\rm d}(p_j,\bar p)+{\rm d}(q,\bar p)\le 2\hat \delta<\delta. 
$$
Noting that $ \sqrt{|\kappa|}\delta<1$ by the choice of $\delta$ in \eqref{jiash0}, one has that
\begin{equation}\label{3.14-12r}
\hbar\left(\sqrt{|\kappa|}{\rm
d}(p_{j},q)\right)\ge   \hbar( \sqrt{|\kappa|}\delta)\ge \hbar (1)>\frac34.
\end{equation}
Recalling that
$F(q)\preceq  F(p_j)$   
and $\sqrt{|\kappa|}t_j\|v_j\|\le\sqrt{|\kappa|}R\|v_j\|\le 1$ by \eqref{jiash0} and \eqref{local-diameter-1}, 
it follows from \eqref{basic-QC-1} (with $t_j$, $v_j$ and $p_j$ in place of $t$, $v$ and $p$) that 
\begin{equation}\label{local-b1lc-ol}
{\rm d}^2(\gamma_j(t_j),q)\leq
{\rm d}^2(p_{j},q)
 +\frac{3t_j^2\|v_j\|^2}{2\hbar\left(\sqrt{|\kappa|}{\rm
d}(p_{j},q)\right)}\leq
{\rm d}^2(p_{j},q)+2Rt_j\|v_j\|^2,
\end{equation}
where the last inequality holds by \eqref{3.14-12r} and  $t_j\le R$. Since $p_{k+1}=\gamma_{k}(t_k)$, it follows  that
\begin{equation}\label{local-b1lc}
{\rm d}^2(p_{k+1},q)
\leq
{\rm d}^2(p_{k},q)+2Rt_k\|v_k\|^2.
\end{equation}
Moreover, we first estimate by \eqref{vk} that
$$
\sum_{l=0}^kt_l\|v_j\|^2\le\sum_{l=0}^k\frac{2(\varphi(p_l)-\varphi(p_{l+1}))}{\beta}=\frac{2(\varphi(p_0)-\varphi(p_{k+1}))}{\beta}\le \frac{
2(\varphi(p_0)-\varphi(q))}{ \beta},$$
where the last inequality holds because $ \varphi(q)\le\varphi(p_{k+1}) $ by  $F(q)\preceq F(p_{k+1})$ and Proposition \ref{phi-p}(ii).
Summing up the  inequalities in \eqref{local-b1lc-ol} over $0\le j\le k-1$, one concludes that
\begin{equation}\label{local-b1lc-0}
{\rm d}^2(p_{k},q)+2Rt_k\|v_k\|^2\leq
{\rm d}^2(p_0,q)+\frac{2R}{\beta}\left(\varphi(p_0)-\varphi(q)\right).
\end{equation}
This, together with \eqref{delta-00ll}, implies that
\begin{equation}\label{local-b1lc-001}
{\rm d}^2(p_{k},q)+2Rt_k\|v_k\|^2\leq
 {\rm d}^2(p_0,q)+c {\rm d}(p_0,q).
\end{equation}
Thus \eqref{lemma1-LC00c} is seen to hold   by  \eqref{local-b1lc}, showing the implication.
%
 The proof is complete.
\end{proof}

Now, we are ready to establish  local convergence of  Algorithm {\rm\ref{DA}} under locally quasi-convex assumption.

\begin{theorem}\label{Local-Convergence}
Let $\bar p\in M$ be such that assumption \eqref{assumption-a} holds. 
Then, for any $\delta>0$, there exist $\bar \delta,\hat \delta>0$ satisfying $\bar \delta<\hat\delta<\frac\delta2$ such that,
for any sequence $\{p_k\}$  generated by Algorithm {\rm\ref{DA}} with initial point $p_0\in\IB(\bar p,\bar\delta)$, if it satisfies
\begin{equation}\label{F-G-B}
\lim_{k\rightarrow+\infty}F(p_k)\succeq F(\bar p),
\end{equation}
then one has the following assertions:

 {\rm (i)}  The sequence $\{p_k\}$ stays in $\IB(\bar p,\hat\delta)$ and converges to a point $p^*$.

 {\rm (ii)}  If it is additionally assumed  that   $\{t_k\}$  has  a positive lower bound or  that $\{t_k\}$  satisfies the Armijo step sizes, then $p^*$ is a critical point of $F$.
\end{theorem}

\begin{proof} 
By the assumed  \eqref{assumption-a}, Lemma \ref{lemma1-LC} is applicable. Thus, for any $\delta>0$, there exist $\bar \delta<\hat\delta<\frac\delta2$ such that,  for  any sequence
   $\{p_{k}\}$   generated by Algorithm {\rm\ref{DA}}, if it satisfies  \eqref{generatedset} then  \eqref{lemma1-LC00c2} holds (for any $k$); hence  the following implication holds for each $k\in\IN$:
   \begin{equation}\label{newjiaim}
   [ \mbox{\eqref{generatedset}   and   \eqref{F-G-B} hold}]\Longrightarrow    p_{k+1}\in \IB(\bar p,  \hat \delta).
   \end{equation}
 Now, let $\{p_k\}$  be a sequence generated by Algorithm {\rm\ref{DA}} with initial point  $p_0\in\IB(\bar p,\bar\delta)$ such that  \eqref{F-G-B} holds. 
Then   one checks by \eqref{newjiaim} (applied to $k=0$) that   $p_1\in \IB(\bar p,\hat\delta)$, and 
   concludes by mathematical induction  that $\{p_k\}\subset\IB(\bar p,\hat\delta)$, 
showing the first conclusion of assertion (i). 
 Consequently,   the sequence $\{p_k\}$ has at least one cluster point, say  $p^*$. Letting $L_{\bar\delta}:=\{p\in \IB(\bar x, \hat\delta): F(p)\preceq \inf_{k\in\IN}F(p_k)\}$, one sees that 
$p^*\in L_{\bar\delta}$ since $\{F(p_k)\}$ is decreasing and $F$ is continuous on $\IB(\bar p,{\hat\delta})$ (using a smaller $\delta$ if necessary). Then, \eqref{lemma1-LC00c} holds for each $q\in L_{\bar \delta}$. Thanks to $\sum_{k=1}^{\infty} t_k\|v_k\|^2<+\infty$ by \eqref{GA-P1}, we get that $\{p_k\}$ is quasi-Fej\'{e}r convergent to $L_{\bar\delta}$. Hence, we conclude by Lemma \ref{fejer2}   that $\lim_{k\rightarrow\infty}p_k=p^*$ (recalling $p^*\in L_{\bar\delta}$). Thus, the second conclusion of assertion (i) is seen to hold. 

Assertion (ii) is a direct consequence of assertion (i) and Proposition \ref{PC}(iii). This completes the proof. 
\end{proof}

\section{Local linear convergence without locally quasi-convex assumption}

To study the linear convergence property, we need the following   Kurdyka-{\L}ojasiewicz-like  property. Let $\bar p$ be a locally weak  Pareto optimum of $F$. Consider the following condition
on some ball $\IB(\bar p, r)$ with some  constant $\alpha>0$:
\begin{equation}\label{local-m-re}
\|v(p)\|^2
\ge \alpha\varphi(p)\quad\mbox{ for each }p\in \IB(\bar p, r),
\end{equation}
 where $v(p)$ is the steepest descent direction at $p$ given by \eqref{Def-vp} and $\varphi$ is defined by \eqref{eq-varp}.
Our second main result in this subsection is on the linear convergence property of Algorithm {\rm\ref{SDAA}}  without locally quasi-convex assumption.
 Note that, to guarantee the linear convergence, it is required  in Theorem \ref{Local-Linear} that the corresponding step sizes $\{t_k\}$ have  a positive lower bound, which is  satisfied  by the Armijo step sizes  in the case when  $JF(\cdot)$ is Lipschitz continuous around $\bar p$; see Lemma \ref{Bt-Lip} below.

\begin{theorem}\label{Local-Linear}
Let $\bar p\in M$ be a 
weak Pareto  optimum of \eqref{MP} such that    
\begin{equation}\label{assumptipn-b01} \mbox{
\eqref{local-m-re} holds for some $\alpha>0$ and $ r>0$}.
\end{equation}
Then, there exists $ \delta>0$ such that, for any   sequence $\{p_k\}$ generated by Algorithm {\rm\ref{SDAA}} with   initial point
      $p_0\in\IB(\bar p,  \delta)$, if   the corresponding step sizes $\{t_k\}$ satisfy   $\b{\rm t}:=\inf_{k\ge 0} \{t_k\}>0$, then $\{p_k\}$ stays at $\IB(\bar p,r)$,
 converges linearly to  a 
 weak  Pareto optimum $p^*$ of \eqref{MP} and satisfies  \begin{equation}\label{L-T}
{\rm d}^2(p_k,p^*)\le \mu  \varphi(p_k) \le \mu\rho^{2k} \varphi  (p_0) \quad\mbox{ for each }k\in\IN,
\end{equation} where $\mu:=\frac{2R}{(1-\rho)^2\beta}$ and   $\rho:=\sqrt{1-\frac{\alpha\beta\b{t}(1-\sqrt{\sigma})^2}{2}}$.
\end{theorem}

\begin{proof}
Note by definition that
$\varphi(\bar p)=0$.
Recalling that $\varphi$ is continuous on $M$, one can
choose $\delta>0$ small enough such that $\delta<r$ and
\begin{equation}\label{delta-0}
\quad\frac{1}{1- {\rho}}\sqrt{\frac{2R\varphi(p)}{\beta}}\le r-\delta\quad\mbox{for all } p\in\IB(\bar p,\delta).
\end{equation}
Below we show that $\delta$ is as desired. To this end, let $\{p_k\}$ be a sequence generated by Algorithm {\rm\ref{SDAA}} with initial point
$p_0\in\IB(\bar p,\delta)$.
Then by     step 3 of Algorithm {\rm\ref{SDAA}} and \eqref{vk}, 
the following relation holds for each $k,l\in\IN$,
 \begin{equation}\label{local-m-L1}
\begin{array}{lll}
{\rm d}^2(p_{k+l+1},p_{k+l}) \le Rt_{k+l}\|v_{k+l}\|^2\le \frac{2R(\varphi(p_{k+l})-\varphi(p_{k+l+1
}))}{\beta}\le \frac{2R\varphi(p_{k+l}) }{\beta} 
\end{array}
 \end{equation}
 (noting $t_{k+l}\in(0,R]$ and $\varphi(p)\ge0$ for all $p\in M$).
 We first show inductively that
\begin{equation}\label{pki}
\{p_k\}\subseteq\IB(\bar p,r).
\end{equation}
Clearly, \eqref{pki} holds   for $k=0$. Now assume that
\begin{equation}\label{pkk}
\{p_j:\; j=0,1,\dots,k\}\subseteq\IB(\bar p,r). 
\end{equation}
Then, it follows from \eqref{Def-gpss} and
 \eqref{assumptipn-b01}   that
 $$
 \|v_j\|^2\ge(1-\sqrt{\sigma})^2\|v(p_j)\|
\ge \alpha (1-\sqrt{\sigma})^2\varphi(p_j)\quad\mbox{for all } j=0,1,\dots,k.
$$
Hence, for all $j=0,1,\dots,k$, one checks from \eqref{vk}   that
\begin{equation}\label{locall-0}
 \varphi(p_{j+1}) \le \varphi(p_j)- \frac{ \beta t_j}2\|v_j\|^2
  \le \varphi(p_j)- \frac{\alpha(1-\sqrt{\sigma})^2\beta t_j}2\varphi(p_j)\le\rho^2 \varphi(p_j).
\end{equation}
Thus, we get that
\begin{equation}\label{locall-0l}
\varphi(p_{j})\le\rho^{2j}\varphi(p_0)\quad\mbox{for all } j=0,1,\dots,k.
\end{equation}
This, together with \eqref{local-m-L1}, implise that  
  for all $j=0,1,\dots,k$,
$$ 
{\rm d}^2(p_{j+1},p_{j}) \le   \rho^{2j}\frac{2R\varphi(p_0)}{\beta}.
$$ 
and so
$$
\begin{array}{lll}
{\rm d}(p_{k+1},\bar p) &\le\sum_{j=0}^k{\rm d}(p_{j+1},p_{j})+{\rm d}(p_{0},\bar p)\le\sum_{j=0}^k\sqrt{\rho^{2j}\frac{2R\varphi(p_0)}{\beta}}+{\rm d}(p_{0},\bar p)\\
&\le\frac{1}{1- {\rho}}\sqrt{\frac{2R\varphi(p_0)}{\beta}
}+{\rm d}(p_{0},\bar p)\le r-\delta+\delta=r,
\end{array}
$$
where the last inequality is by the choice of $\delta$ (see \eqref{delta-0}). Thus, \eqref{pki} is valid by mathematical induction. Furthermore, by the arguments for proving \eqref{pki}, we see that \eqref{locall-0} and \eqref{locall-0l} hold for all $j\in \IN$.
Hence the following relations hold for each $k,l\in\IN$:
\begin{equation}\label{local-m-L1k}
\begin{array}{lll}
\varphi(p_{k+l})\le\rho^{2l}\varphi(p_k)\quad \mbox{and}\quad\varphi(p_{k}) \le\rho^{2k}\varphi(p_0). 
\end{array}
\end{equation}
Recalling $\varphi(p_k)\ge0$ for each $k$, there holds 
\begin{equation}\label{varphi0}
\lim_{k\rightarrow\infty}\varphi(p_k)=0.
\end{equation}
 Combing \eqref{local-m-L1} and  \eqref{local-m-L1k} yields  that
$$
\begin{array}{lll}
{\rm d}(p_{k+l+1},p_{k+l})\le \rho^l \sqrt{\frac{2R \varphi(p_{k})}{\beta}} \quad\mbox{for any }k,l\in\IN,
\end{array}
$$
  and then 
 \begin{equation}\label{local-m-L1-wang-2}
\begin{array}{lll}
{\rm d}(p_{k+l},p_k)\le\sum_{j=1}^l{\rm d}(p_{k+j},p_{k+j-1})\le \frac{ 1-\rho^l }{1-\rho}  \sqrt{\frac{2R \varphi(p_{k})}{\beta}}.
\end{array}
 \end{equation}
Thus, in view of \eqref{varphi0}, the sequence $\{p_k\}$ is a Cauchy sequence, and then $\{p_k\}$ converges to some point $p^*$ satisfying $\varphi(p^*)=0$ (noting that   $\varphi$ is continuous), and so $p^*$ is a weak  Pareto optimum of \eqref{MP}.  
Letting $l$ goes to infinite in \eqref{local-m-L1-wang-2} and noting the second item of \eqref{local-m-L1k}, we have that
$$
{\rm d}(p_{k},p^*)\le \frac{ 1 }{1-\rho} \sqrt{\frac{2R \varphi(p_{k})}{\beta}}\le\frac{ 1 }{1-\rho} \sqrt{\frac{2R\varphi(p_{0})}{\beta}}\rho^{  k}.
$$
Hence, \eqref{L-T} is seen to hold, completing the proof.
\end{proof}

\begin{remark}\label{extendgra}
Theorem \ref{Local-Linear} establishs  the linear convergence property of Algorithm {\rm\ref{SDAA}}  without locally quasi-convex assumption, which seems new even in linear spaces setting. Furthermore, in the case when the multiobjective optimization is reduced to scalar optimization (i.e., $I=\{1\}$), our result  improves sharply the corresponding result in \cite{WLWYSIAM2019} in the sense that we remove  the local  quasi-convexity  assumption.
\end{remark}

The following lemma provides a sufficient condition for the step size sequence $\{t_k\}$ generated by the Armijo step sizes to have a positive lower bound.

\begin{lemma}\label{Bt-Lip}
Let $\bar p\in M$ be such that 
assumption \eqref{assumption-a} holds,
and suppose that $JF(\cdot)$ is Lipschitz continuous around $\bar p$. 
Then,     there exist $\b{\rm t}>0$ and $\bar\delta>0$      such that, for any $p_0\in\IB(\bar p,\bar\delta)$,  if Algorithm \ref{SDAA}  employs
the Armijo step sizes and the generated   sequence $\{p_k\}$  satisfies \eqref{F-G-B}, 
then     the generated step sizes $\{t_k\}$  satisfies that $\inf_{k\in\IN}t_k\ge \b{\rm t}$. 
\end{lemma}


\begin{proof}
By assumption, 
Theorem \ref{Local-Convergence} is applicable to getting that, for any $\delta>0$,  there exist $\bar\delta,\hat\delta>0$  satisfying $\bar \delta<\hat\delta<\frac \delta2$  with the property stated there.
Without loss of generality, we may assume further that $3\nu^{-1}\hat\delta<r_{\rm cvx}({\bar p})$, and 
  there exists $L>0$ such that for each $i\in I$,
\begin{equation}\label{Lip-1-G}
\|\nabla f_i(p)-P_{p,q}\nabla f_i(q)\|\le L{\rm d}(p,q)\quad\mbox{for any }p,q\in\IB(\bar p,3\nu^{-1}\hat\delta)
\end{equation}
(where $\nu$ is chosen by the Armijo step size rule \eqref{GDA-2-A}).

Let $\b{t}:=\min\left\{\nu, \frac{\nu(1-\beta) }{2 L}\right\}$. Below, we show that $\b{t},\hat\delta$ are as desired. To do this,
let   $p_0\in\IB(\bar p,\bar\delta)$, and let  $\{t_k\}$ and $\{p_k\}$ be the generated Armijo step sizes and
the generated sequence by Algorithm {\rm\ref{SDAA}} with initial point $p_0$, respectively.
Now
fix $k$  and assume that $t_k\le\nu$. Then, by \eqref{GDA-2-A}, we see that there exists $i\in I$ such that
\begin{equation}\label{Lip-3-Bt}
f_i(\gamma_{k}(\nu^{-1}t_{k}))- f_i(p_{k})\ge  \nu^{-1}\beta t_{k}\langle\nabla f_i(p_{k}),v_k\rangle.
\end{equation}
Noting that $\IB(\bar p, \bar c\bar \delta)$ is strongly convex, one sees that $\gamma_k({[0,t_k]})$ is the unique minimal geodesic joining $p_k$ to $p_{k+1}$. Therefore  $t_k\|v_k\|={\rm d}(p_k,p_{k+1}))$, and it follows that
\begin{equation*}\label{mingeo}  {\rm d}(p_{k},\gamma_{k}({\nu^{-1}}  t_k))\le \nu^{-1}t_k\|v_k\|= \nu^{-1}{\rm d}(p_k,p_{k+1})\le 2\nu^{-1}\hat\delta,\end{equation*}
(see Theorem \ref{Local-Convergence}(i) for the last inequality). Thus, using the triangle inequality  and noting that $1<\nu^{-1}$, one checks that  $\gamma_{k}({\nu^{-1}}  t_k) \in \IB(\bar p,3\nu^{-1}\hat\delta)$ because
\begin{equation*}\label{geodes-tr}
{\rm d}(\bar p,\gamma_{k}({\nu^{-1}}  t_k)) \le {\rm d}(\bar p,p_{k})+\nu^{-1}{\rm d}(p_k,p_{k+1})  \le 3\nu^{-1}\hat\delta.\end{equation*}
Using the mean value theorem, 
we can choose  $\bar t_k\in(0,t_k)$ to satisfy that
\begin{equation}\label{Full-p5}
f_i(\gamma_{k}({\nu^{-1}}t_k))-f_i(p_{k})=\left\langle\nabla f_i\left(\gamma_{k}({\nu^{-1}}\bar t_k)\right), {\nu^{-1}}t_kP_{\gamma_k,\gamma_{k}({\nu^{-1}}\bar t_k),p_k} v_k\right\rangle\\
\end{equation}
Since
$$
\begin{array}{lll}
&&\left\langle\nabla f_i\left(\gamma_{k}({\nu^{-1}}\bar t_k)\right), P_{\gamma_k,\gamma_{k}({\nu^{-1}}\bar t_k),p_k} v_k\right\rangle\\
&=& \left\langle P_{\gamma_k,p_k,\gamma_{k}({\nu^{-1}}\bar t_k)}\nabla f_i\left(\gamma_{k}({\nu^{-1}}\bar t_k)\right)-\nabla f_i(p_{k}),v_k\right\rangle+\langle\nabla f_i(p_{k}),v_k\rangle\\
&\le&\|P_{\gamma_k,p_k,\gamma_{k}({\nu^{-1}}\bar t_k)}\nabla f_i\left(\gamma_{k}({\nu^{-1}}\bar t_k)\right)-
\nabla f_i(p_{k})\|\cdot\|v_k\|+\langle\nabla f_i(p_{k}),v_k\rangle\\
&\le& \nu^{-1} t_kL \|v_k\|^2+\langle\nabla f_i(p_{k}),v_k\rangle,
\end{array}
$$
where the last inequality holds by \eqref{Lip-1-G} (as $\gamma_{k}({\nu^{-1}}\bar t_k) \in \IB(\bar p,3\nu^{-1}\hat\delta)$),
it follows from \eqref{Full-p5}   that
$$
f_i(\gamma_{k}({\nu^{-1}}t_k))-f_i(p_{k})
 \le {\nu^{-1}}t_k(\nu^{-1} t_kL \|v_k\|^2+\langle\nabla f_i(p_{k}),v_k\rangle).
$$
Combining this and \eqref{Lip-3-Bt}, we conclude that
$$0\le \nu^{-1} t_kL \|v_k\|^2+(1-\beta)\langle\nabla f_i(p_{k}),v_k\rangle.$$
Hence, it follows from \eqref{Def-gps} that
$$0\le  \left(\nu^{-1} t_kL+(1-\beta)\left(-\frac12\right)\right)\|v_k\|^2.$$ This implies that
 $t_k\ge \frac{\nu(1-\beta) }{ 2L}$ (in the case when   $t_k\le \nu$), and so
 $\inf_{k\in\IN}t_k\ge \min\left\{\nu, \frac{\nu(1-\beta) }{ 2L}\right\}$ as desired to show.
\end{proof}

\section{Global  convergence}
The following theorem regards the global convergence and the linear convergence of Algorithm \ref{SDAA}. We emphasize that the convergence result as well as the linear convergence rate of Algorithm \ref{SDAA} is independent of the curvatures of $M$.

\begin{theorem}\label{full-1}
Suppose that the sequence $\{p_k\}$ generated by Algorithm \ref{SDAA}  has a cluster point $\bar p$. 
  Then, the following assertions hold:

{\rm (i)} 
 If  \eqref{assumption-a} holds, then $\{p_k\}$ converges to $\bar p$.

{\rm (ii)} If $\bar p $ is a 
weak Pareto  optimum of \eqref{MP},  $\inf_{k\ge 0} \{t_k\}>0$ and 
assumption \eqref{assumptipn-b01} holds, 
then
$\{p_k\}$ converges linearly to $\bar p$. 
\end{theorem}

\begin{proof}
Noting that \eqref{F-G-B} is naturally satisfied as $\{F(p_k)\}$ is non-increasing monotone and $\bar p$ is a cluster point, we get from Theorem \ref{Local-Convergence}(i) 
that there exists $\delta>0$ such that any sequence generated by Algorithm \ref{SDAA} with initial point in $\IB(\bar p,\delta)$ is convergent. Now $\bar p$ is a cluster point, so there exists some $k_0\in \IN$ such that $p_{k_0}\in \IB(\bar p,\delta)$. Thus, $\{p_k\}$ converges to some point, which in fact equals to $\bar p$ and assertion (i) holds.

With a similar argument that we did for assertion (i), but using Theorem \ref{Local-Linear} 
instead
of Theorem \ref{Local-Convergence}(i), one sees that assertions (ii) holds. 
The proof is complete.
\end{proof}

The following lemma provides some sufficient conditions ensuring the boundedness of the sequence $\{p_k\}$ generated by Algorithm \ref{SDAA} (and so the existence of a cluster point).
Set
$$\mathcal{L}_0:=\{p\in M:F(p)\preceq F(p_0)\}.$$

\begin{lemma}\label{lemma-BX}
Let $\{p_k\}$ be a sequence generated by Algorithm \ref{SDAA} with initial point $p_0$. Then, $\{p_k\}$ is bounded 
provided one of the assumptions {\rm (a)} and {\rm (b)} holds:

{\rm (a)} $\mathcal{L}_0$ is bounded.

{\rm (b)}  $\mathcal{L}_0$ is totally convex with  its curvatures  being  bounded from below and $F$ is quasi-convex  on   $\mathcal{L}_0$ (e.g., $F$ is quasi-convex  on $M$ and $M$ is of lower bounded curvatures).
\end{lemma}


\begin{proof}
Note that $\{p_k\}\subseteq \mathcal{L}_0$ as $\{F(p_k)\}$ is non-increasing monotone. Then, $\{p_k\}$ is clear bounded under assumption {\rm (a)} .
Under assumption (b), with a similar argument as in the proof for \cite[Theorem 3.7]{Wangxm2019}, one can check that $\{p_k\}$ is   bounded.
\end{proof}

The following corollary is immediate from Theorem \ref{full-1}  and Lemma \ref{lemma-BX}. Particularly, 
the global convergence result (assertion (i)) under assumption {\rm (b)} in Lemma \ref{lemma-BX}  extends the corresponding one in \cite[Theorem 3.7]{Wang2010}   which was established   for the case when  Algorithm \ref{SDAA} employs the Armijo step sizes (noting that in this case    any cluster point $\bar p$  of a generated  sequence    satisfies \eqref{assumption-a} by Proposition \ref{PC}(iii)). As for assertion (ii), as far as we know, it is new even  in the linear space  setting.

\begin{corollary}\label{full-AG}
Suppose that one of assumptions {\rm (a)}  and {\rm (b)}  in Lemma \ref{lemma-BX} holds.
Then, any sequence $\{p_k\}$ generated by Algorithm \ref{SDAA}   has at least a cluster point $\bar p$; furthermore, if $\bar p$ satisfies \eqref{assumption-a}, then assertions {\rm (i)} and {\rm (ii)} in Theorem \ref{full-1} hold.
\end{corollary}

\begin{Acknowledgments} 
Research of the 
author was supported in part by the National Natural Science
Foundation of China (grant numbers 11661019,  11771397).
\end{Acknowledgments}

\end{document}